\documentclass[12pt, reqno]{amsart}
\usepackage{amsmath, amsthm, amscd, amsfonts, amssymb, graphicx, color}
\usepackage[bookmarksnumbered, colorlinks, plainpages]{hyperref}
\hypersetup{colorlinks=true,linkcolor=red, anchorcolor=green, citecolor=cyan, urlcolor=red, filecolor=magenta, pdftoolbar=true}

\newtheorem{theorem}{Theorem}[section]
\newtheorem{lemma}[theorem]{Lemma}

\newtheorem{cor}[theorem]{Corollary}
\theoremstyle{definition}

\theoremstyle{remark}
\newtheorem{remark}[theorem]{\bf{Remark}}
\numberwithin{equation}{section}
\begin{document}
		\title[Weighted numerical radius inequalities for operator and operator matrices ]  {Weighted numerical radius inequalities for operator and operator matrices }
		\author[R.K. Nayak]{  Raj Kumar Nayak}

		\address{(Nayak) Department of Mathematics, SRM University-AP, Amaravati, India}
		\email{rajkumarju51@gmail.com}
	\date{}
	\maketitle 
		\begin{abstract}
		The concepts of weighted numerical radius has been defined in recent times. In this article, we obtain several upper bound for weighted numerical radius of operators and $2 \times 2$ operator matrices which generalize and improves some well known famous inequality for classical numerical radius. We also obtain an upper bound for the weighted numerical radius of the Aluthge transformation, $\tilde{T}$ of an operator $T \in \mathcal{B}(\mathcal{H}),$ where $\tilde{T} = |T|^{1/2} U |T|^{1/2}$ and $T = U |T|$ be the canonical polar decomposition of $T.$

	\end{abstract}
		\section{Introduction}
	Let $\mathcal{B}(\mathcal{H})$ be the $C^*-$algebra of all bounded linear operator on a Hilbert  space $\mathcal{H}$. Cartesian decomposition of an operator $T \in \mathcal{B}(\mathcal{H})$, can be written as $T = \Re(T) + i \Im(T)$ where $\Re(T) = \frac{T + T^*}{2}$ and $\Im(T) = \frac{T-T^*}{2i}$. Generalization of this decomposition was introduced in \cite{SKS}, called weighted real and imaginary parts of $T$ defined as :  \[\Re_t(T)= (1-t)T^* + tT~~\mbox{and}~\Im_t(T) =\frac{(1-t)T - tT^*}{i}~~\mbox{for}~t\in [0,1] .\] Clearly for $t=\frac{1}{2}$, $\Re_t(T) = \Re(T)$ and $\Im_t(T) = \Im(T)$. It is easy to observe that for every operator $T \in \mathcal{B}(\mathcal{H})$, $\Re_t(T) + i\Im_t(T) = (1-2t)T^* + T.$ In \cite{CSM}, the weighted numerical radius was defined by \[w_t(T) = \sup_{\theta \in \mathbb{R}} \left\| \Re_t(e^{i\theta} T)\right\| .\] Conde et. al. in \cite{CSM} introduced weighted numerical radius in the following way: 
	\begin{eqnarray*}
	\mbox{For}~T \in \mathcal{B}(\mathcal{H})~~\mbox{and}~~t \in [0,1],&&\\ 
w_t(T) &= &\sup_{\|x\|=1} \left|\langle\left(\Re_t(T) + i \Im_t(T) \right)x, x \rangle \right|
= w\left((1-2t)T^* + T \right).
	\end{eqnarray*}
Similarly weighted operator norm of $T \in \mathcal{B}(\mathcal{H})$, is defined by \[\|T\|_t = \sup_{\|x\| = \|y\|=1} \left|\langle \left(\Re_t(T) + i \Im_t(T) \right)x, y \rangle   \right|. \] For further study on weighted numerical radius we refer the reader \cite{ SKS, CSM, Zamani2022}.
	
	We can easily observe that $\|T\|_t = \|(1-2t)T^* + T\|.$ Similar to the classical numerical radius inequality, weighted numerical radius also satisfy triangle inequality, i.e.  for $X, Y \in \mathcal{B}(\mathcal{H})$,\[ w_t(X+Y) \leq w_t(X) + w_t(Y).\] It also satisfy unitary equivalence property, i.e. for $T \in \mathcal{B}(\mathcal{H}),$ \[ w_t(U^*TU) = w_t(T).\] One can easily observe that for $t=\frac{1}{2}$, \[w_t(T) = w(T) ~\mbox{and}~\|T_t\| = \|T\|.\]
	For $T \in \mathcal{B}(\mathcal{H})$, $w(T)$ is a norm on $\mathcal{B}(\mathcal{H}),$ equivalent to the operator norm that satisfy the inequality, 
\begin{equation} \label{basic}
	\frac{\|T\|}{2}  \leq w(T) \leq \|T\|.
\end{equation}

	Various refinements of inequality \ref{basic} has been obtained by famous mathematicians in recent times  such as \cite{F} follows: \[\frac{1}{4} \|T^*T + TT^*\| \leq w^2(T) \leq \frac{1}{2} \|T^*T + TT^*\| .\]
	Pintu et. al. in \cite[Th. 2.1]{PLAA2021} prove that for $T \in \mathcal{B}(\mathcal{H})$, \[w(T) \geq \frac{1}{2} \|T\| +\frac{1}{2} \left|\|\Re (T)\| - \|\Im (T)\| \right| . \]
	For $T \in \mathcal{B}(\mathcal{H}),$ the Aluthge transformation \cite{Yamazaki} of $T$, denoted as $\tilde{T}$ is defined by $\tilde{T} = |T|^{1/2} U |T|^{1/2}$ where $|T| = (T^*T)^{1/2}$ and $U$ is the partial isometry associated with the polar decomposition of $T$ and so $T = U |T|$, $Ker T = Ker U$.\\
	\noindent using Aluthge transform, Yamazaki in \cite[Th 2.1]{Yamazaki} proved that if $T \in \mathcal{B}(\mathcal{H}),$ then \[w(T) \leq \frac{1}{2} \left( \|T\| + w(\tilde{T})\right). \]
We have generalized this inequality for weighted numerical radius. We have also generalized and improves the weighted numerical radius inequalities for $2 \times 2$ operator matrices.
		\section{Weighted Inequality of Operators}
		We begin this section with few well known inequalities which will be essential to prove our theorems. We start with Buzano inequality. 
		\begin{lemma}\label{buzano}\cite{Buzano}
			Let $x, y, e \in \mathcal{H}$ with $\|e\|=1.$ Then 
		\[|\langle x, e \rangle \langle e, y \rangle| \leq \frac{1}{2} \left(|\langle x, y \rangle| + \|x\|\|y\| \right).\]
		\end{lemma}
	Next lemma is for positive operator.
	\begin{lemma} \label{lemmasimon}
	Let $T \in \mathcal{B}(\mathcal{H})$ be a positive operator. Then for all $r \geq 1$ and $x \in \mathcal{H}$ with $\|x\|=1$, then  \[\langle Tx, x \rangle^r \leq \langle T^r x, x \rangle. \]	
	\end{lemma}
	Next we state generalized mixed Schwarz inequality.
\begin{lemma}\label{mixed schwarz} \cite[Th. 1]{Furuta}
	If $T \in \mathcal{B}(\mathcal{H})$, then for all $x, y \in \mathcal{H}$ and $\alpha \in [0,1]$ \[|\langle Tx, y \rangle|^2 \leq \langle |T|^{2\alpha} x, x \rangle \langle |T^*|^{2(1-\alpha)}  y, y \rangle.\]
	
\end{lemma}
Our first theorem of this section follows as: 
	\begin{theorem} \label{th2}
		Let $T \in \mathcal{B}(\mathcal{H})$ and $t \in [0,1]$. Then
		\[ w_t^2(T) \leq (1-2t)^2 w^2(T) + (1-2t) w(T^2) + (1-t) \|T^*T + TT^*\|.\]
		
	\end{theorem}	
		\begin{proof}
			Let $x \in \mathcal{B}(\mathcal{H})$ with $\|x\|=1$. Then we get, 
			\begin{eqnarray*}
			|\langle \left((1-2t)T^* + T\right)x, x \rangle | &\leq& (1-2t) |\langle T^*x, x \rangle| + |\langle Tx, x \rangle|\\
			\Rightarrow |\langle \left((1-2t)T^* + T\right)x, x \rangle |^2 &=& (1-2t)^2 |\langle Tx, x \rangle|^2 + |\langle Tx, x \rangle|^2\\ &+& 2(1-2t) |\langle Tx, x \rangle  \langle x, T^*x \rangle| \\
			&\leq& (1-2t)^2 |\langle Tx, x \rangle|^2 + \langle |T|x, x \rangle \langle |T^*|x, x \rangle\\& +& (1-2t) \left[|\langle Tx, T^*x \rangle| + \|Tx\|\|T^*x\| \right] \\ && \,\,\,\,\,\,\,\,\,\,\,\,\,\,\,\,\,\,\,\,\,\,\,\,\,\,\,\,\,\,\,\,\,\,\,\,\,\,\,\,\,\,\,\,\,\,\,\,\,\,\,\mbox{(using Lemma \ref{buzano} and \ref{mixed schwarz})}\\
			&\leq& (1-2t)^2 |\langle Tx, x \rangle|^2 + \frac{1}{2} \langle \left(|T|^2 + |T|^2 \right)x, x \rangle\\ &+& (1-2t) |\langle T^2x, x \rangle| + \frac{1}{2} (1-2t) \langle \left( |T|^2 + |T^*|^2\right) x, x \rangle \\
			&\leq& (1-2t)^2 w^2(T) + (1-2t)w(T^2)\\ &+& (1-t) \|T^*T + TT^*\|
			\end{eqnarray*}
		Now taking supremum over $x$ with $\|x\|=1$ we get our required bound.
		\end{proof}
		
		\begin{cor}
			If we take $t = \frac{1}{2}$ in Theorem \ref{th2} we get the inequality proved in\\ \cite[Th. 1]{F}, which says that \[w^2(T) \leq \frac{1}{2} \|T^*T + TT^*\|. \]
			So the inequality obtained in Theorem \ref{th2} generalizes \cite [Th. 2]{F}.
		\end{cor}
	Next we prove an inequality involving weighted norm.	
		
		\begin{lemma}\label{t-product}
			Let $T, S \in \mathcal{B}(\mathcal{H})$ and $t \in [0, 1]$, then 
			\[\|TS\|_t^2 \leq (2-4t + 4t^2) \|TS\|^2 + (1-2t) w((TS)^2) + (1-2t) w((S^*T^*)^2) . \]
		\end{lemma}
		\begin{proof}
			Let $x \in \mathcal{H}$ with $\|x\|=1$. By simple calculation we get,
			\begin{eqnarray*}
			 \left\|\left((1-2t)(TS)^* + TS \right)\right)x\|^2
			&=& \left \langle [(1-2t)(TS)^* + TS ]x, [(1-2t)(TS)^* + TS ]x \right \rangle\\
			&=& (1-2t)^2 \|(TS)^*x\|^2 + (1-2t) \langle (TS)^*x, TSx \rangle\\ &&+ (1-2t) \langle TSx, (TS)^*x \rangle + \|TSx\|^2\\
			&\leq& (2-4t+4t^2) \|TS\|^2 + (1-2t) w((TS)^2) \\&&+  (1-2t) w((S^*T^*)^2)
			\end{eqnarray*}
		Now taking supremum over $x$ with $\|x\|=1$ we get our required inequality.
		\end{proof}
	Next we give an upper bound for weighted numerical radius which improves\\ the second inequality in (\ref{basic}).
		\begin{theorem}\label{th3}
			Let $T \in \mathcal{B}(\mathcal{H})$ and $t \in [0, 1]$. Then \[w_t^2(T) \leq (1-2t + 2t^2) \|TT^* + T^*T\| + (1-2t) w\left(T^2 + (T^*)^2\right). \]
		\end{theorem}
	\begin{proof}
		Let $x \in \mathcal{H}$ with $\|x\| =1$. By using Lemma \ref{mixed schwarz} we get,
		\begin{eqnarray*}
		|\langle [(1-2t)T^* + T]x, x \rangle|^2 &\leq& \langle |(1-2t)T^* + T|x, x\rangle \langle |(1-2t)T + T^*|x, x \rangle\\
		&\leq& \frac{1}{2} \left[ \left\langle |(1-2t)T^* + T|^2x,x \right\rangle + \left\langle |(1-2t)T + T^*|^2x, x \right \rangle \right]\\ && \,\,\,\,\,\,\,\,\,\,\,\,\,\,\,\,\,\,\,\,\,\,\,\,\,\,\,\,\,\,\,\,\,\,\,\,\,\,\,\,\,\,\,\,\,\,\,\,\,\,\,\,\,\,\,\,\,\,\,\,\,\,\,\,\,\,\,\,\,\,\,\,\,\,\,\,\,\,(\mbox{using Lemma \ref{lemmasimon}})\\
		&=& \frac{1}{2}\left[(1-2t)^2 \langle (TT^* + T^*T)x, x \rangle +  \langle (TT^* + T^*T)x, x \rangle\right] \\ &&+ (1-2t) \langle (T^2 + (T^*)^2)x, x \rangle \\ 
		&\leq& (1-2t + 2t^2) \|TT^* + T^*T\| + (1-2t) w \left(T^2 + (T^*)^2 \right)
		\end{eqnarray*}
	Taking supremum over $x$ with $\|x\|=1$ we get our desired inequality.
	\end{proof}	
		\begin{remark}
			Let $T \in \mathcal{B}(\mathcal{H})$ and $t \in \frac{1}{2}$. Then from Theorem \ref{th3} we get
			\begin{eqnarray*}
				w_t^2(T) &\leq& (1-2t + 2t^2) \|TT^* + T^*T\| + (1-2t) \left\|T^2 + (T^*)^2\right\|\\ 
				&\leq& (1-2t + 2t^2) 2\|T\|^2 + (1-2t) 2\|T\|^2\\&& \,\,\,\,(\mbox{using inequality~ (\ref{basic}) and numerical radius power inequality})\\ &=& 4(1-2t + t^2) \|T\|^2 .
			\end{eqnarray*}
		So clearly for $t = \frac{1}{2}$, the inequality improves the second inequality of (\ref{basic}).
		\end{remark}
	In the following theorem we give another improvement of second inequality of (\ref{basic}) for $t = \frac{1}{2}.$
	\begin{theorem}\label{th4}
		Let $T \in \mathcal{B}(\mathcal{H})$ and $0\leq t \leq 1$. Then \[w_t^2(T) \leq (2-4t+4t^2)w^2(T) + (1-2t) w(T^2) + \frac{1}{2}(1-2t) \|TT^* + T^*T\|.\]
		
	\end{theorem}
\begin{proof}
	Let $x \in \mathcal{H}$ with $\|x\|=1$. Then using Lemma \ref{buzano} we get,
	\begin{eqnarray*}
	\left|\langle ((1-2t)T^* + T)x, x \rangle \right|^2  &\leq& \left((1-2t) |\langle T^*x, x \rangle| + |\langle Tx, x \rangle| \right)^2\\
	&=& (1-2t)^2 |\langle Tx, x \rangle|^2 + |\langle Tx, x \rangle|^2 \\&&+ 2(1-2t) |\langle Tx, x \rangle \langle x, T^*x \rangle|\\
	&\leq& (2-4t+4t^2) |\langle Tx, x \rangle|^2 + (1-2t) \left[\langle Tx, T^*x \rangle\right]\\&& + \|Tx\|\|T^*x\| \\
	&\leq& (2-4t + 4t^2)w^2(T) + (1-2t) w(T^2)\\ && + \frac{1}{2}(1-2t) \|TT^* + T^*T\|.
 	\end{eqnarray*}
 Now taking supremum over $x$ with $\|x\|=1$ we get our desired inequality.
\end{proof}
\begin{remark}
	Let $T \in \mathcal{B}(\mathcal{H})$ and $t \in [0,1]$. Then using second inequality\\ of (\ref{basic}) and power inequality we get
	\begin{eqnarray*}
		w_t^2(T) &\leq& (2-4t+4t^2) \|T\|^2 + (1-2t) \|T\|^2 + \|T\|^2\\
		&=& 4(1-2t + t^2) \|T\|^2.
	\end{eqnarray*}
Clearly inequality obtained in Theorem \ref{th4}~ improves the second inequality in (\ref{basic})\\ for $t =\frac{1}{2}$.
\end{remark}
Next we state Generalized Polarization Identity which will be useful to prove our final weighted numerical radius inequality of this section.
\begin{lemma}\label{Polarization}
	Let $T \in \mathcal{B}(\mathcal{H})$ and $x, y \in \mathcal{H}$. Then 
	\begin{eqnarray*}
	\langle Tx, y \rangle &=& \frac{1}{4} \left[\langle T(x+y), (x+y) \rangle - \langle T(x-y), (x-y) \right]\\
	&& + \frac{i}{4}\left[ \langle T(x+iy), (x+iy) \rangle - \langle T(x-iy), (x-iy) \rangle \right].
	\end{eqnarray*}
\end{lemma}
\begin{theorem}
	Let $T \in \mathcal{B}(\mathcal{H})$ and $t \in [0,1]$ then 
	\[w_t(T) \leq (1-t)\left(\|T\| + w(\tilde{T})\right). \]
\end{theorem}
		\begin{proof}
			Let $x \in \mathcal{H}$ with $\|x\|=1$. Assume that $T = U|T|$ be the polar decomposition of $T$. Then for every $\theta \in \mathbb{R}$ we have,
			\begin{eqnarray*}
			\Re \langle e^{i\theta} \left[(1-2t)T^* + T \right]x, x \rangle &=& (1-2t) \Re \langle e^{i\theta} T^*x, x \rangle + \Re \langle e^{i\theta} Tx, x\rangle\\
			&=& (1-2t) \Re \langle e^{i\theta} |T|U^*x, x \rangle + \Re \langle e^{i\theta} U|T|x, x\rangle\\
			&=& (1-2t) \Re \langle e^{-i\theta} |T|x, U^*x \rangle + \Re \langle e^{i\theta} |T|x, U^*x\rangle\\&&\,\,\,\,\,\,\,\,\,\,\,\,\,\,\,\,\,\,\,\,\,\,\,\,\,\,\,\,\,\,\,\,\,\,\,\,\,\,\,\,\,\,\,\,\,\,\,\,\,\,\,\,\,\,\,\,\,\,\,\,\,\,\,\,\,\,\,\,\,\,\,\,\,\,\,\,\,\,\,\,\ (\mbox{since $\Re z = \Re \overline{z}$})\\
		&=&	\frac{1-2t}{4} \langle |T|(e^{-i\theta} + U^*)x, (e^{-i\theta} + U^*)x \rangle\\&&-  \frac{1-2t}{4}\langle |T|(e^{-i\theta} - U^*)x, (e^{-i\theta} - U^*)x \rangle\\&& + \frac{1}{4}\langle |T|(e^{i\theta} + U^*)x, (e^{i\theta} + U^*)x \\&&- \langle |T| (e^{i\theta} - U^*)x, (e^{i\theta} - U^*)x ~~  (\mbox{using Lemma \ref{Polarization}})\\
		&=& \frac{1-2t}{4}\langle (e^{i\theta} + U) |T| (e^{-i\theta} + U^*)x, x \rangle \\&&- \frac{1-2t}{4} \langle (e^{i\theta} - U) |T| (e^{-i\theta} - U^*) x, x \rangle \\ &&+ \frac{1}{4} \langle (e^{-i\theta} + U) |T| (e^{i\theta} + U^*)x, x \rangle \\ &&- \frac{1}{4}  \langle (e^{-i\theta} - U) |T| (e^{i\theta} - U^*)x, x \rangle\\
		&\leq& \frac{1-2t}{4} \left\|(e^{i\theta} + U) |T| (e^{-i\theta} + U^*) \right\| \\ &&+ \frac{1}{4} \left\|(e^{-i\theta} + U) |T| (e^{i\theta} + U^*) \right\|\\
		&=& \frac{1-2t}{4} \left\||T|^{1/2} (e^{-i\theta} + U^*)(e^{i\theta} +U) |T|^{1/2} \right\| \\ && + \frac{1}{4} \left\||T|^{1/2} (e^{i\theta} +U^*) (e^{-i\theta} +U) |T|^{1/2} \right\|\\&&\,\,\,\,\,\,\,\,\,\,\,\,\,\,\,\,\,\,\,\,\,\,\,\,\,\ (\mbox{using the fact that $\|SS^*\| = \|S^*S\|$})
			\end{eqnarray*}
	\begin{eqnarray*}
		&=& \frac{1-2t}{4} \left\|2|T| + e^{-i\theta} |T|^{1/2} U |T|^{1/2} + e^{i\theta} |T|^{1/2} U^* |T|^{1/2} \right\|\\ && + \frac{1}{4} \left\|2|T| + e^{i\theta} |T|^{1/2} U |T|^{1/2} + e^{-i\theta} |T|^{1/2} U^* |T|^{1/2} \right\|\\
		&=& \frac{1-2t}{2} \left\||T| + \Re (e^{-i\theta} \tilde{T}) \right\| + \frac{1}{2} \left\||T| + \Re (e^{i\theta} \tilde{T}) \right\|\\
		&\leq& \frac{1-2t}{2} \left(\|T\| + w(\tilde{T}) \right) + \frac{1}{2} \left(\|T\| + w(\tilde{T}) \right)\\
		&=&  (1-t)\left(\|T\| + w(\tilde{T})\right)
			\end{eqnarray*}
		\end{proof}
		
	\begin{cor}
		If we take $t = \frac{1}{2}$ then we will get the following famous result \\proved by Yamazaki \cite{Yamazaki}, for $T \in \mathcal{B}(\mathcal{H})$ \[w(T) \leq \frac{1}{2} \left(\|T\| + w(\tilde{T}) \right) .\]
	\end{cor}

\section{Weighted Numerical Radius of Operator Matrices}
		
		We begin this section with the following lemmas for weighted numerical radius \\of operator matrices, the proofs of which are very trivial.
		\begin{lemma}\label{weighted op}
			Let $X, Y \in \mathcal{B}(\mathcal{H})$ and $t \in [0,1]$. Then the following equalities holds:
			\begin{eqnarray*}
				&&(i)~ w_t \left(\begin{array}{cc}
				X&O\\
				O&Y	
			\end{array} \right) = \max \{w_t(X), w_t(Y) \}\\
		&&(ii)~~w_t \left(\begin{array}{cc}
			O&X\\
			X&O	
		\end{array} \right) = w_t(X)\\
	&&(iii)~~w_t \left(\begin{array}{cc}
		X&Y\\
		Y&X	
	\end{array} \right) = \max \{w_t(X+Y), w_t(X-Y) \}\\
&& (iv)~~w_t \left(\begin{array}{cc}
	O&X\\
	Y&O	
\end{array} \right) = \sup_{\theta \in \mathbb{R}}\frac{1}{2} \left\|(1-2t)(e^{-i\theta} X + e^{i\theta} Y^*) + (e^{i\theta}X + e^{-i\theta}Y^*) \right\|.
			\end{eqnarray*}
		\end{lemma}
		\begin{theorem} \label{th op1}
				Let $X, Y \in \mathcal{B}(\mathcal{H}) $ and $t \in [0,1]$. Then
				\begin{eqnarray*}
				 w_t^2\left(\begin{array}{cc}
				O&X\\
				Y&O	
				\end{array} \right) &\leq&  (1-2t+2t^2) \left\||X|^2 + |Y^*|^2 \right\| +(2-4t+4t^2) w(YX)\\ 
			&&	+ (1-2t) \left(\|X\|^2 + \|YX\| + \|Y\|^2 \right) . 
				\end{eqnarray*}
				\end{theorem}
			\begin{proof}
		 Let $\theta \in \mathbb{R}$. Assume that $S = (1-2t)(e^{-i\theta} X + e^{i\theta} Y^*) + (e^{i\theta}X + e^{-i\theta}Y^*).$\\ By simple calculation it follows: 
			\begin{eqnarray*}
			S^*S &=&  (1-2t)^2 \left(|X|^2 + |Y^*|^2  +2 \Re\left( e^{-2i\theta} YX\right) \right) \\
			&& + (1-2t) \left( 2\Re\left(e^{2i\theta}\right) |X|^2  + 2\Re\left(e^{2i\theta}\right)|Y|^2  +2\Re(YX) \right)\\
			&& + \left(|X|^2 + |Y^*|^2 + 2\Re \left(e^{2i\theta} YX \right)\right).
			\end{eqnarray*}
		Now taking norm on both side and using triangle inequality we have,
		\begin{eqnarray*}
		\frac{1}{2}\|S\|^2 &\leq& (1-2t+2t^2) \left\||X|^2 + |Y^*|^2 \right\| +(1-2t)^2 \left\|\Re(e^{-2i\theta} YX) \right\|\\
		&& +  \left\|\Re(e^{2i\theta} YX) \right\| \\
		&&+ (1-2t) \left(\left\|\Re(e^{2i\theta})|X|^2 \right\| + \left\|\Re(e^{-2i\theta}) |Y^*|^2\right\| + \left\|\Re(YX) \right\| \right)\\
		&\leq& (1-2t +2t^2) \left\||X|^2 + |Y^*|^2 \right\| + (2-4t+4t^2)w(YX)\\
		 && + (1-2t) \left( \||X|^2\| + \||Y^*|^2\| + \|YX\|\right)\\
		&=& (1-2t+2t^2) \left(\left\||X|^2 + |Y^*|^2 \right\|\right) + (2-4t+4t^2)w(YX)\\ 
		&&+ (1-2t) \left(\|X\|^2 + \|YX\| + \|Y\|^2 \right).
	\end{eqnarray*}
Now taking supremum over $\theta \in \mathbb{R}$ and using Lemma \ref{weighted op} (iv) we get our \\required inequality.
	\end{proof}
The following corollary can be obtained from Theorem \ref{th op1}.
\begin{cor}
		If we take $t=\frac{1}{2}$ in Theorem  \ref{th op1} we get the following\\ result \cite[Th. 2.4]{OK} which says that, for $T \in \mathcal{B}(\mathcal{H})$
	\[ w^2(T) \leq \frac{1}{4} \|T^*T + TT^*\| + \frac{1}{2} w(T^2). \] Thus the inequality obtain in Theorem \ref{th op1} generalizes the upper bound obtained\\ in  \cite[Th. 2.4]{OK}.
\end{cor}
Next we prove the following theorem.
\begin{theorem} \label{th op2}
	Let $X, Y, Z, W \in \mathcal{B}(\mathcal{H})$ and $t \in [0,1]$. Then 
	\begin{eqnarray*}
		w_t^2  \left( \begin{array}{cc}
			X&Y\\
			Z&W
		\end{array}\right) &\leq& 8(1-t)^2 \max\{w^2(X), w^2(W) \} \\&&+ 2(1-t)^2 \max \{\left\| |Z|^2 + |Y^*|^2 \right\|, \left\||Y|^2 + |Z^*|^2 \right\| \}\\&& + 4(1-t)^2 \max \{w(YZ), w(ZY) \}.
	\end{eqnarray*}
\end{theorem}
		\begin{proof}
			Let $T = \left(\begin{array}{cc}
				X&Y\\
				Z&W
			\end{array} \right)$, $R = \left(\begin{array}{cc}
			X&O\\
			O&W
		\end{array} \right)$ and $S = \left(\begin{array}{cc}
			O&Y\\
			Z&O
		\end{array} \right) $.\\ So, $T = R + S.$ Assume that $x \in \mathcal{H} \oplus \mathcal{H}$ with $\|x\| =1.$ Then 
	\begin{eqnarray*}
		\left|\langle \left((1-2t)T^* + T \right)x, x \rangle  \right|^2 &\leq& 
		 4(1-t)^2 \left(|\langle Rx, x \rangle| + |\langle Sx, x \rangle| \right)^2\\
		 &\leq& 8(1-t)^2 \left(|\langle Rx, x \rangle|^2 + |\langle Sx, x \rangle|^2 \right)\\&&\,\,\,\,\,\,\,\,\,\,\,\,\,\,\,\,\,\,\,\,\,\,\,\,\,\,\,~\mbox{By convexity of $f(t) =t^2$}\\
		 &=& 8(1-t)^2 \left(|\langle Rx, x \rangle|^2 + |\langle Sx, x \rangle \langle x, S^*x \rangle| \right)\\
		 &\leq& 8(1-t)^2|\langle Rx, x \rangle|^2 + 4(1-t)^2|\langle Sx, S^*x \rangle| \\&&+ 4 (1-t)^2 \|Sx\|\|S^*x\|~~,(\mbox{using Lemma \ref{buzano}})\\
		 &\leq& 8(1-t)^2 |\langle Rx, x \rangle|^2 +4(1-t)^2 |\langle S^2x, x \rangle|\\ && + 2(1-t)^2 \langle (|S|^2 + |S^*|^2)x, x \rangle \\
		 &=& 8(1-t)^2 \left|\left\langle \left( \begin{array}{cc}
		 	X&O\\
		 	O&W
		 \end{array}\right)x, x \right\rangle\right|^2\\&+&  4(1-t)^2 \left|\left\langle \left( \begin{array}{cc}
		 YZ&O\\
		 O&ZY
	 \end{array}\right)x, x \right\rangle\right|\\ &+&  2(1-t)^2 \left| \left\langle\left( \begin{array}{cc}
	 |Z|^2 + |Y^*|^2&O\\
	 O& |Y|^2 + |Z^*|^2
 \end{array}\right)x, x \right \rangle\right|\\
&\leq& 8(1-t)^2 \max \{w^2(X), w^2(W) \}\\&+&  4(1-t)^2 \max \{w(YZ), w(ZY) \} \\&+& 2(1-t)^2 \max \{\left\||Z|^2 + |Y^*|^2 \right\|, \left\||Y|^2 + |Z^*|^2 \right\|  \}.
	\end{eqnarray*}
		\end{proof}
	\begin{cor} \label{cor1}
		If we take $X=W$ and $Y=Z$ in Theorem \ref{th op2} then we get 
		\begin{eqnarray*}
			w_t^2 \left(\begin{array}{cc}
				X&Y\\
				Y&X
			\end{array} \right) &=& \max \{w_t^2(X+Y), w_t^2(X-Y)\} \\
		&\leq& 8(1-t)^2 w^2(X) + 4(1-t)^2 w(Y^2) + 2(1-t)^2 \left\||Y|^2 + |Y^*|^2\right\| .
		\end{eqnarray*}
	\end{cor}	
		\begin{remark} Taking $X = Y$ in corollary \ref{cor1} we get
			\begin{eqnarray*}
				w_t^2(X) &\leq& 2(1-t)^2 w^2(X) + (1-t)^2 w(X^2) + \frac{(1-t)^2}{2} \left\| |X|^2 + |X^*|^2 \right\|.
			\end{eqnarray*}
		In particular if we take $t = \frac{1}{2}$, then we obtain the following well known inequality: \[w^2(X) \leq \frac{1}{2} w(X^2) + \frac{1}{4} \left\| |X|^2 + |X^*|^2 \right\| .\]
			
		\end{remark}
		Next we prove another upper bound for $2 \times 2$ matrices.
		\begin{theorem}\label{th op3}
				Let $X, Y, Z, W \in \mathcal{B}(\mathcal{H})$ and $t \in [0,1]$. Then 
			\begin{eqnarray*}
				w_t^2  \left( \begin{array}{cc}
					X&Y\\
					Z&W
				\end{array}\right) &\leq& 4(1-t)^2 \max \{w^2(X), w^2(W) \} + 4(1-t)^2  w \left(\begin{array}{cc}
				O&YW\\
				ZX & O
			\end{array} \right)\\ && + 2(1-t)^2 \max \{w(YZ), w(ZY) \}\\ && + (1-t)^2 \max \left\{\left\|2|X|^2 + 3 |Y^*|^2 + |Z|^2 \right\|, \left\|2|W|^2 + 3 |Z^*|^2 + |Y|^2 \right\| \right\}.
			\end{eqnarray*}
		\end{theorem}
	\begin{proof}
		Let $T, R, S$ as in Theorem \ref{th op2} and $x \in \mathcal{H} \oplus \mathcal{H}$ with $\|x\| =1.$ Then 
		\begin{eqnarray*}
				\left|\langle \{(1-2t)T^*+ T \}x, x \rangle  \right|^2 &\leq& 4(1-t)^2 \left(|\langle Rx, x \rangle| + |\langle Sx, x \rangle | \right)^2\\
				&=& 4(1-t)^2 \left[|\langle Rx, x \rangle| ^2 + 2 |\langle Rx, x \rangle| |\langle Sx, x \rangle| + |\langle Sx, x \rangle|^2  \right]\\
				&=& 4(1-t)^2 \left[|\langle Rx, x \rangle|^2 + 2 |\langle Rx, x \rangle \langle x, S^*x \rangle|\right]\\ && +4(1-t)^2 |\langle Sx, x \rangle \langle x, S^*x \rangle| \\
				&\leq& 4(1-t)^2 |\langle Rx, x \rangle|^2 + 4(1-t)^2|\langle Rx, S^*x \rangle|\\&& + 4(1-t)^2\|Rx\| \|S^*x\| \\ &+&  2(1-t)^2 \left[|\langle Sx, S^*x \rangle| + \|Sx\|\|S^*x\| \right]\\&&\,\,\,\,\,\,\,\,\,\,\,\,\,\,\,\,\,\,\,\,\,\,\,\,\,\,\,\,\,\,\,\,\,\,\,\,\,\,\,\,\,\,\,\,\,\,\,\,\,\,\,\,\,\,\,\,\,\,\,\,\,\,\,\,\,\,\,\,\,\,\,\,\,\,\,\,\,\,\ ~(\mbox{using Lemma \ref{buzano}}) \\
				&\leq& 4(1-t)^2 |\langle Rx, x \rangle|^2 + 4(1-t)^2 |\langle (SR)x, x \rangle|\\ &+&  2(1-t)^2 \left\langle \left(|R|^2 + |S^*|^2 \right)x, x \right \rangle + 2(1-t)^2 |\langle S^2x, x \rangle| \\ &+&  (1-t)^2 \left\langle \left(|S|^2 + |S^*|^2 \right)x, x \right \rangle	\\
				&\leq& 4(1-t)^2 \max \{w^2(X), w^2(W) \}\\&& + 4(1-t)^2 w \left(\begin{array}{cc}
					O&YW\\
					ZX&O
				\end{array} \right) \\ &+&  2(1-t)^2 \max \{w(YZ), w(ZY) \}\\&+&  (1-t)^2\max\left\{ \|W\|, \|V\| \right\}
				\end{eqnarray*}
			where, $W =2|X|^2 + 3 |Y^*|^2 + |Z|^2$~ and~ $V =2|W|^2 + 3 |Z^*|^2 + |Y|^2$.
	
	\end{proof}	
		\begin{cor} \label{cor 2}
			If we take $X=W$ and $Y=Z$ in Theorem \ref{th op3}, then we get 
			\begin{eqnarray*}
				w_t^2 \left( \begin{array}{cc}
					X&Y\\
					Y&X
				\end{array}\right) &=& \max \{w_t^2(X+Y), w_t^2(X-Y) \}\\
			&\leq& 4(1-t)^2 \left(w^2(X) + w(YX)\right)+ 2(1-t)^2 w(Y^2)\\
			&& + (1-t)^2 \left\|2|X|^2 + 3 |Y^*|^2 + |Y|^2 \right\|.
			\end{eqnarray*}
		\end{cor}
		\begin{remark}
			Taking $X = Y$ in corollary \ref{cor 2} we get 
			\begin{eqnarray*}
				4w_t^2(X) &\leq& 4(1-t)^2 \left(w^2(X) + w(X^2) \right) \\
				&& + 2(1-t)^2 w(X^2) + 3(1-t)^2 \left\||X|^2 + |X^*|^2 \right\|\\
				&=& 4(1-t)^2 w^2(X) + 6 (1-t)^2 w(X^2) + 3 (1-t)^2 \left\||X|^2 + |X^*|^2 \right\|.
			\end{eqnarray*}
			In particular if we take $t = \frac{1}{2}$ we get the  following well known inequality \[w^2(X) \leq \frac{1}{2} w(X^2) + \frac{1}{4} \left\| |X|^2 + |X^*|^2\right\|. \]
		\end{remark}
		Next we state  following lemma which is the extension of Buzano inequality.
	\begin{lemma} \label{ext buzano}\cite[Lemma 3.1]{BF}
		Let $x, y, e \in \mathcal{H}$ with $\|e\| =1$. The for $\alpha \in [0, 1]$, \[\left|\langle x, e \rangle \langle e, y \rangle \right|^2 \leq \frac{1}{4} \left[(1+\alpha) \|x\|^2\|y\|^2 + (3-\alpha) \|x\|\|y\| |\langle x, y \rangle| \right]. \]
	\end{lemma}
Using the above lemma we prove our next theorem.

		\begin{theorem}\label{th op4}
				Let $X, Y, Z, W \in \mathcal{B}(\mathcal{H})$ and $t \in [0,1]$. Then 
			\begin{eqnarray*}
				w_t^4  \left( \begin{array}{cc}
					X&Y\\
					Z&W
				\end{array}\right) &\leq& 128(1-t)^4 \max\{w^4(X), w^4(W) \}\\
			&& + 16(1-t)^4 (1+\alpha) \max \{\left\||Z|^4 + |Y^*|^4 \right\|, \left\||Y|^4 + |Z^*|^4 \right\| \}\\
			&&+ 16(1-t)^4 (3-\alpha) \max \{\left\||Z|^2 + |Y^*|^2\right\|, \left\||Y|^2 + |Z^*|^2 \right\| \}\\&& \times \max \{w(YZ), w(ZY) \}.
			\end{eqnarray*}
		\end{theorem}
		\begin{proof}
			Let $T, R, S $ as in Theorem \ref{th op2} and $x \in \mathcal{H} \oplus \mathcal{H}$ with $\|x\| =1$. Then 
			\begin{eqnarray*}
				\left| \left\langle [ (1-2t) T^* + T]x, x \right\rangle \right|^4 &\leq& \left[(1-2t)|\langle T^*x, x \rangle| + |\langle Tx, x \rangle| \right]^4\\
				&=& 16 (1-t)^4 |\langle Tx, x \rangle|^4\\
				&\leq& 16(1-t)^4 \left(|\langle Rx, x \rangle| + |\langle Sx, x \rangle| \right)^4\\
				&\leq& 16\cdot 8 (1-t)^4 \left(|\langle Rx, x \rangle|^4 + |\langle Sx, x \rangle|^4 \right)\\
				&&\,\,\,\,\,\,\,\,\,\,\,\,\,\,\,\,\,\,\,\,\,\,\,(\mbox{using the convexity of $f(t) = t^4$})\\
				&=& 128 (1-t)^4 \left(|\langle Rx,x \rangle|^4 + |\langle Sx, x \rangle \langle x, S^*x \rangle|^2 \right)\\
				&\leq& 128 (1-t)^4 |\langle Rx, x \rangle|^4 \\
				&& + 32(1-t)^4 (1+\alpha) \|Sx\|^2\|S^*x\|^2\\&& +32(1-t)^4 (3-\alpha) \|Sx\|\|S^*x\| |\langle Sx, S^*x \rangle| \\
				&&\,\,\,\,\,\,\,\,\,\,\,\,\,\,\,(\mbox{using Lemma \ref{ext buzano}})\\
				&\leq& 128(1-t)^4 |\langle Rx, x \rangle|^4\\ && + 16(1-t)^4 (1+\alpha) \left\langle \left(|S|^4 + |S^*|^4\right)x, x \right\rangle \\ && + 16(1-t)^4 (3-\alpha)  \left\langle \left(|S|^2 + |S^*|^2\right)x, x \right\rangle \times |\langle S^2x, x \rangle|\\
				&=& 128 (1-t)^2 \left|\left\langle\left( \begin{array}{cc}
					X&O\\
					O&W
				\end{array}\right)x,x \right\rangle \right|\\
			&+&  16(1-t)^4 (1+\alpha) \left\langle\left( \begin{array}{cc}
				|Z|^4 + |Y^*|^4&O\\
				O&|Y|^4 + |Z^*|^4
			\end{array}\right)x,x \right\rangle\\&+&  16 (1-t)^4 (3-\alpha) \left\langle\left( \begin{array}{cc}
			|Z|^2 + |Y^*|^2	&O\\
			O&	|Z|^2 + |Y^*|^2
		\end{array}\right)x, x \right\rangle\\&& \times \left|\left\langle\left( \begin{array}{cc}
		YZ&O\\
		O&ZY
	\end{array}\right)x,x \right\rangle	 \right|.
			\end{eqnarray*}
		Now taking supremum over $x$ with $\|x\|=1$ we get our desired inequality.
		\end{proof}
		\begin{cor}\label{cor 3}
			If we take $X = W$ and $Y = Z$ in Theorem \ref{th op4} we get 
			\begin{eqnarray*}
				w_t^2 \left( \begin{array}{cc}
					X&Y\\
					Y&X
				\end{array}\right) &=& \max \{w_t^4(X+Y), w_t^2(X-Y) \}\\
			&\leq& 128(1-t)^4 w^4(X) + 16(1-t)^4 (1+\alpha) \left\||Y|^4 + |Y^*|^4 \right\|\\
			&& + 16(1-t)^4 (3-\alpha) \left\||Y|^2 + |Y^*|^2 \right\|w(Y^2).
			\end{eqnarray*}
		\end{cor}
		\begin{remark}
			Taking $X = Y$ and $t= \frac{1}{2}$ we get the following result which was\\ proved by Bani-Domi et. al. in \cite[Remark 3.2]{BF}
			\[w^4(X) \leq \frac{1+\alpha}{8} \left\||X|^4 + |X^*|^4 \right\| + \frac{3-\alpha}{8} \left\||X|^2 + |X^*|^2 \right\| w(X^2). \]
		\end{remark}
		For the next theorem we need the following lemma which was proved by\\ Bani-Domi et. al. in \cite[Lemma 3.2]{BF}.
		\begin{lemma} \label{Buzano gen}
			Let $x, y, e \in \mathcal{H}$ with $\|e\| =1$ and $\alpha \in [0,1].$ Then 
			\begin{eqnarray*}
			\left|\langle x, e \rangle \langle e, y \rangle \right|^2 &&\leq \frac{\alpha}{4} \left(\|x\|^2\|y\|^2 + 2 \|x\|\|y\| |\langle x, y \rangle| + |\langle x, y \rangle|^2 \right)\\
			&& + \frac{1-\alpha}{2} \left(\|x\|\|y\| + |\langle x, y \rangle| \right) \left|\langle x, e \rangle \langle e, y \rangle \right|.
			\end{eqnarray*}
		\end{lemma}

		\begin{theorem} \label{th op5}
			Let $X, Y, Z, W \in \mathcal{H}$ and $t, \alpha \in [0,1].$ Then 
			\begin{eqnarray*}
				w_t^4 \left(\begin{array}{cc}
					X&Y\\
					Z&W
				\end{array} \right) &\leq& 128 (1-t)^4 \max \{w^4(X), w^4(W) \} \\
			&& + 16\alpha (1-t)^4 \max \{\left\||Z|^4 + |Y^*|^4 \right\|, \left\||Y|^4 + |Z^*|^4 \right\| \} \\ && + 32\alpha(1-t)^4 \max\{w(YZ), w(ZY) \}\\&&\times \max \{\left\||Z|^2 + |Y^*|^2 \right\|,\left\| |Y|^2 + |Z^*|^2\right\| \} \\ && + 32\alpha (1-t)^4 \max \{w^2(YZ), w^2(ZY) \}\\&&+ 32(1-\alpha)(1-t)^4 \max \left\{\left\||Z|^2 + |Y^*|^2 \right\|, \left\| |Y|^2 + |Z^*|^2\right\| \right\} \\&&\times  w^2\left(\begin{array}{cc}
				O&Y\\
				Z&O
			\end{array} \right)\\ && + 64(1-\alpha) (1-t)^4 \max \{w(YZ), w(ZY) \} w^2\left(\begin{array}{cc}
			O&Y\\
			Z&O
		\end{array} \right).
			\end{eqnarray*}
		\end{theorem}
			\begin{proof}
				Let $T, R, S$ are as in Theorem \ref{th op2}. Assume that $x \in \mathcal{H} \oplus \mathcal{H}$ with $\|x\| =1.$ Then 
				\begin{eqnarray*}
					\left|\langle \{(1-2t)T^* + T \}x, x \rangle \right|^4 &=& 16 (1-t)^4 |\langle Tx, x \rangle|^4\\
					&\leq& 16(1-t)^4 \left( |\langle Rx, x \rangle| +  |\langle Sx, x \rangle| \right)^4\\
					&\leq& 128 (1-t)^4  \left( |\langle Rx, x \rangle|^4 +  |\langle Sx, x \rangle|^4 \right)\\&&\,\,\,\,\,\,\,\,\,\,\,\,\,\,\,\,\,\,\,\,\,\,\,\,~(\mbox{using convexity of $f(t) = t^4$})\\
					&=&  128 (1-t)^4  \left( |\langle Rx, x \rangle|^4 +  |\langle Sx, x \rangle \langle x , S^*x \rangle|^2 \right) \\
					&\leq& 128(1-t)^4 |\langle Rx, x \rangle|^4 \\
					&& + 32 \alpha (1-t)^4 \left(\|Sx\|^2\|S^*x\|^2 + 2 \|Sx\|\|S^*x\| |\langle Sx, S^*x \rangle|\right)\\&& + 32 \alpha (1-t)^4 |\langle Sx, S^*x \rangle |^2 \\
					&&+ 64(1-t)^4(1-\alpha) \left(\|Sx\|\|S^*x\| + |\langle Sx, S^*x \rangle| \right) |\langle Sx, x \rangle|^2 
						\end{eqnarray*}
				\begin{eqnarray*}
					&\leq& 128 (1-t)^4 |\langle Rx, x \rangle|^4 \\
					&& + 16\alpha (1-t)^4 \left\langle (|S|^4 + |S^*|^4)x, x   \right\rangle\\&& + 32\alpha (1-t)^4 |\langle S^2x, x \rangle| \left\langle (|S|^2 + |S^*|^2)x, x  \right\rangle\\
					&& + 32\alpha (1-t)^4 |\langle S^2x, x \rangle|^2\\&& + 32(1-\alpha)(1-t)^4 \left(\left\langle (|S|^2 + |S^*|^2)x,x  \right\rangle+ 2 |\langle S^2x, x \rangle  |\right)\\ && \times |\langle Sx, x \rangle|^2\\ 
					&=& 128(1-t)^4 \left|\left\langle \left( \begin{array}{cc}
						X&O\\
						O&W
					\end{array}\right)x,x \right\rangle \right|^4 \\ && + 16\alpha (1-t)^4 \left|\left\langle \left( \begin{array}{cc}
					|Z|^4 + |Y^*|^4&O\\
					O&|Y|^4 + |Z^*|^4
				\end{array}\right)x,x \right\rangle \right| \\&&+ 32\alpha(1-t)^4 \left|\left\langle \left( \begin{array}{cc}
				YZ&O\\
				O&ZY
			\end{array}\right)x,x \right\rangle \right|\\&& \times \left|\left\langle \left( \begin{array}{cc}
			|Z|^2 + |Y^*|^2&O\\
			O&|Y|^2 + |Z^*|^2
		\end{array}\right)x,x \right\rangle \right|  \\&&+ 32 \alpha (1-t)^4 \left|\left\langle \left( \begin{array}{cc}
		YZ&O\\
		O&ZY
	\end{array}\right)x,x \right\rangle \right|^2\\ && + 32(1-\alpha)(1-t)^4 \left|\left\langle \left( \begin{array}{cc}
	|Z|^2 + |Y^*|&O\\
	O&|Y|^2 + |Z^*|^2
\end{array}\right)x,x \right\rangle \right|  \\ &&\times \left|\left\langle \left( \begin{array}{cc}
O&Y\\
Z&O
\end{array}\right)x,x \right\rangle \right|^2\\&& + 64(1-\alpha)(1-t)^4 \left|\left\langle \left( \begin{array}{cc}
YZ&O\\
O&ZY
\end{array}\right)x,x \right\rangle \right| \\ && \times \left|\left\langle \left( \begin{array}{cc}
O&Y\\
Z&O
\end{array}\right)x,x \right\rangle \right|^2\\
&\leq& 128 (1-t)^4 \max \{w^4(X), w^4(W) \} \\
&& + 16\alpha (1-t)^4 \max \{\left\||Z|^4 + |Y^*|^4 \right\|, \left\||Y|^4 + |Z^*|^4 \right\| \} \\ && + 32\alpha(1-t)^4 \max\{w(YZ), w(ZY) \}\\ && \times \max \{\left\||Z|^2 + |Y^*|^2 \right\|,\left\| |Y|^2 + |Z^*|^2\right\| \} \\
 &&+ 32\alpha (1-t)^4 \max \{w^2(YZ), w^2(ZY) \}\\&&+ 32(1-\alpha)(1-t)^4 \max \{\left\||Z|^2 + |Y^*|^2 \right\|, \left\| |Y|^2 + |Z^*|^2\right\| \} \\&&\times w^2\left(\begin{array}{cc}
	O&Y\\
	Z&O
\end{array} \right)\\ && + 64(1-\alpha) (1-t)^4 \max \{w(YZ), w(ZY) \} w^2\left(\begin{array}{cc}
	O&Y\\
	Z&O
\end{array} \right).
				\end{eqnarray*}
			Now taking supremum over $x$ with $\|x\|=1$ we get our required inequality.
\end{proof}
By taking $X = W$ and $Y = Z$ in Theorem \ref{th op5} we get the following corollary.
\begin{cor} \label{cor 4}
(i)~	If $X, Y \in \mathcal{B}(\mathcal{H}).$ Then 
	\begin{eqnarray*}
		w^4 \left(\begin{array}{cc}
			X&Y\\
			Y&X
		\end{array} \right) &=& \max \{w^4(X+Y), w^4(X-Y) \}\\
	&\leq& 128 (1-t)^4 w^4(X) + 16 \alpha (1-t)^4 \left\||Y|^4 + |Y^*|^4\right\|\\
	&& + 32 \alpha (1-t)^4 w(Y^2) \||Y|^2 + |Y^*|^2\|
	 + 32 \alpha (1-t)^4 w^2(Y^2)\\
	&& + 32 (1-\alpha) (1-t)^4 w^2(Y) \left\||Y|^2 + |Y^*|^2 \right\|\\
	&& + 64 (1-\alpha) (1-t)^4 w(Y^2) w^2(Y).
	\end{eqnarray*}
(ii)~If we take $t = \frac{1}{2}$ in Theorem \ref{th op5} then we get the inequality proved by Bani-Domi et. al in \cite[Th. 3.3]{BF}.
\end{cor}

			\bibliographystyle{amsplain}
				
		\end{document}